\documentclass[10pt,reqno]{amsart}

\usepackage{amsmath}
\usepackage{amssymb}
\usepackage[title]{appendix}
\usepackage[initials,msc-links]{amsrefs}
\usepackage{hyperref}
\usepackage{todonotes}

\newtheorem*{theorem*}{Theorem}
\newtheorem{theorem}{Theorem}
\newtheorem{lemma}{Lemma}
\newtheorem{conjecture}{Conjecture}
\newtheorem{observation}{Observation}

\newtheorem{question}{Question}

\title[Statistics of the divisor graph] {Counting primitive subsets and other statistics of the divisor graph of $\{1,2, \ldots, n\}$}
\author{Nathan McNew}
\address{Department of Mathematics, Towson University 8000 York Road, Towson, MD 21252}
\email{nmcnew@towson.edu}
\begin{document}
\begin{abstract}
    Let $Q(n)$ denote the count of the primitive subsets of the integers $\{1,2\ldots, n\}$. We give a new proof that $Q(n) = \alpha^{(1+o(1))n}$ for some constant $\alpha$,  which allows us to give a good error term and to improve upon the lower bound for the value of $\alpha$.  We also show that the method developed can be applied to many similar problems related to the divisor graph, including other questions about primitive sets, geometric-progression-free sets, and the divisor graph path-cover problem.
\end{abstract}

\maketitle

\section{Previous work on counting primitive sets}

A set of integers is called a \textit{primitive set} if no integer in the set is a divisor of another.  For example the prime numbers form a primitive set, as do the integers with exactly $k$ prime factors for any fixed $k$.  
    
In 1990, as part of a paper \cite{camerd} filled with conjectures and problems related to subsets of the integers with various properties, {Cameron}  and {Erd\H{o}s} considered the counting function $Q(n)$ of the number of primitive subsets of the integers up to $n$ (sequence \href{http://oeis.org/A051026}{A051026} in the OEIS \cite{oeis}).  They note the bounds $ (\sqrt{2})^n  < Q(n) < 2^n$ can be observed immediately from the observation that for fixed $n$, the set of integers $\left \{ \left \lfloor \frac{n}{2} \right\rfloor+1,\left \lfloor \frac{n}{2} \right\rfloor +2,\ldots, n\right \}$ is primitive (a set of size $\lceil \frac{n}{2}\rceil$), as are any of the $2^{\lceil\frac{n}{2}\rceil}$ subsets of this set.
    
Based on this observation, they conjecture that $\displaystyle{\lim_{n\to \infty} Q(n)^{1/n}}$ exists. In addition to this conjecture, they outline a proof that the bounds above can be improved to \begin{equation}
     1.55967^n \ll Q(n) \ll 1.60^n. \label{bds:ErdCam}
\end{equation}

Many of the conjectures in \cite{camerd} have attracted a substantial amount of attention, especially those of an additive nature.  More recently some of the multiplicative questions have begun to attract attention.  Recently, their conjecture above about the count of primitive sets was proven by Angelo \cite{angelo}.
\begin{theorem}[Angelo, 2017] \label{thm:angelo}
     The limit $\displaystyle{\lim_{n\to \infty} Q(n)^{1/n}}$ exists.  Equivalently, there exists a constant $\alpha$ such that $Q(n) = \alpha^{(1+o(1))n}$. 
\end{theorem}
   
The proof proceeds by considering subsets of the integers $[1,n]$ in which the ratio between any two included elements is not an $s$-smooth integer (an integer without any prime factors larger than $s$) for fixed $s$, and then allowing $s$ to tend to infinity.  Unfortunately the proof was not effective in the sense that it doesn't give a way to improve the bounds in \eqref{bds:ErdCam}.

More recently Vijay \cite{vijay} considered the problem of counting primitive subsets of $[1,n]$ of maximum size, called \textit{maximum primitive sets}.  A pigeonhole principle argument shows that for any $n$ the maximum size of a primitive subset of the integers up to $n$ is $\lceil \frac{n}{2}\rceil$.  He considers only even $n$ (sequence \href{http://oeis.org/A174094}{A174094} in the OEIS), although the result extends immediately to all integers $n$.  In particular he shows that the count $M(n)$ of the number of primitive subsets of $[1,n]$ of size $\lceil \frac{n}{2}\rceil$ satisfies the bounds 
\[1.303^n \ll M(2n) \ll 1.408^n\]
or, in terms of $n$,
\[1.141^n \ll M(n) \ll 1.187^n.\]

These questions have also recently been investigated independently by Liu, Pach and Palincza \cite{lpp}, and some of the results of that paper overlap with the results presented here.  While many of the ideas are similar, we take a slightly different approach that gives an error term and computationally seems to provide better lower bounds, though is not so good at giving numerical upper bounds.  They present a method to effectively compute $\alpha$, and improve the bounds in \eqref{bds:ErdCam} to 

\[1.571068^n \ll Q(n) \ll 1.574445^n.\]
They also consider maximal primitive sets, and show that the limit $\lim_{n \to \infty} M(n)^{1/n}$ exists, (which we denote $\beta$)  show that this constant can be effectively computed, and improve the bounds above to 
\[1.148172^n \ll M(n) \ll 1.148230^n \text{ \hspace{1cm} or \hspace{1cm}}  1.3183^n \ll M(2n) \ll 1.31843^n.\]

\section{Results and other applications}

We show that these two constants $\alpha$ and $\beta$ can be given explicitly in terms of an infinite product described in Section \ref{sec:countprim}. This observation allows us to prove the following.

\begin{theorem} \label{thm:alpha}
    For any $\epsilon>0$, and $n$ sufficiently large, the number of primitive subsets of $[1,n]$ is \[Q(n) = \alpha^{n\left(1+O_\epsilon\left(\exp\left(-(1-\epsilon)\sqrt{\log n \log \log n}\right)\right)\right)}\]
    or equivalently,
    \[\log Q(n) = n\log \alpha + O_\epsilon\left(n\exp\left(-(1-\epsilon)\sqrt{\log n \log \log n}\right)\right)\] 
    and the constant $\alpha$ is effectively computable.
    \end{theorem}
    
This result is a corollary of our main theorem, which can be used to compute a wide variety of statistics about subsets of $\{1,2,\ldots, n\}$ with certain multiplicative structure.   Before we can state our main theorem however we must introduce a definition. Consider the divisor graph of a set of integers, the graph obtained by treating each number as a vertex and connecting two integers by an edge if one divides another.  In this context, for example, a primitive subset would correspond to an independent set of vertices.  

We also consider the divisor graphs of the integers in an interval $[a,n]$.  In general these graphs won't be connected, and we will be interested only in the component of this graph connected to the vertex $a$.  We say that a function $f(a,n)$ \textit{depends only on the connected component of $a$ in the divisor graph of the interval $[a,n]$} if $f(a,n)=f(b,m)$ whenever the connected component of $b$ in the divisor graph of $[b,m]$ is isomorphic\footnote{Here we allow any isomorphism of graphs, however in practice throughout the paper we only ever consider the linear isomorphisms obtained by multiplying each element by a fixed constant.  One could instead restrict the definition to such linear isomorphisms and the main theorem would still apply, allowing the result to be applied to a potentially wider range of functions $f(a,n)$.} to that of $a$ in $[a,n]$ (with the vertex of $a$ corresponding to $b$ in the isomorphism).  Because of this, smooth numbers (numbers without large prime factors) play an important role throughout the paper.  We denote by $P^+(n)$ the largest prime divisor of $n$ (and take the convention that $P^+(1)=1$), and by $P^-(n)$ the smallest prime divisor of $n$.  
With this, we can state our main theorem.

\begin{theorem}[Main Theorem] \label{thm:main} Suppose $\epsilon>0$, $A\geq0$ and $f(a,n)$ is a bounded function $|f(a,n)|\leq A$ that depends only on the connected component of $a$ in the divisor graph of the interval $[a,n]$.  Then there exists a constant \[C_f = \sum_{i=1}^\infty  \sum_{\substack{d\\ P^+(d)\leq i}} \sum_{t \in [id,(i+1)d)} \left(\frac{f(d,t)}{t(t+1)}\prod_{p\leq i} \frac{p-1}{p}\right)\] 
such that \[\sum_{a=1}^n f(a,n) = nC_f + O_\epsilon\left(An\exp\left(-(1-\epsilon)\sqrt{\log n \log \log n}\right)\right).\]
\end{theorem}
As we will see, many questions about subsets with multiplicative structure can be regarded as questions about this divisor graph, and we demonstrate several different applications of this result.  First we apply it to several related problems about primitive sets. For example, we obtain the corresponding result for counting maximum primitive subsets as well. 
\begin{theorem} \label{thm:beta}
The limit $\lim_{n \to \infty} M(n)^{1/n} = \beta$ exists.  Furthermore, for any $\epsilon>0$ and $n$ sufficiently large, 
\[M(n) = \beta^{n\left(1+O_\epsilon\left(\exp\left(-(1-\epsilon)\sqrt{\log n \log \log n}\right)\right)\right)}\] and the constant $\beta$ is effectively computable.
\end{theorem}

We also use it to count maximal primitive subsets of $[1,n]$ (sequence \href{http://oeis.org/A326077}{A326077} in OEIS), where a primitive subset is maximal if no additional integer from the interval can be added to the set without having one integer divide another in the subset. Note that all maximum primitive subsets are maximal, but not all maximal primitive subsets are a maximum primitive set.  For example, $\{2,3,5,7\}$ is a maximal primitive subset of the integers up to $9$, but not a maximum primitive set, as $\{5,6,7,8,9\}$ is larger.  We obtain analogues of these theorems for $m(n)$, the count of the maximal primitive subsets of $[1,n]$.

\begin{theorem} \label{thm:eta}
The limit $\lim_{n \to \infty} m(n)^{1/n} = \eta$ exists.  Furthermore, for any $\epsilon>0$ and $n$ sufficiently large,  
\[m(n) = \eta^{n\left(1+O_\epsilon\left(\exp\left(-(1-\epsilon)\sqrt{\log n \log \log n}\right)\right)\right)}\] and the constant $\eta$ is effectively computable.
\end{theorem}

By estimating the relevant constant from Theorem \ref{thm:main} we can numerically approximate each of these constants to arbitrary precision.  In Section \ref{sec:numerics} we describe the computations performed to compute partial sums of this product, which give us bounds for the size of these constants.  

\begin{theorem} \label{thm:numalpha}
The constants $\alpha$, $\beta$, and $\eta$  satisfy the bounds \[1.572939<\alpha\]
\[1.148205<\beta \hspace{3mm} (\text{or } 1.318376<\beta^2 ) \]
\[1.230163<\eta < 1.257843.\]
\end{theorem}

The lower bounds for $\alpha$ and $\beta$ are stronger than those  appearing in \cite{lpp} however our method has not yet been able to improve upon the upper bounds obtained by their methods.  Combining their result with ours gives the new known ranges 
\[ 1.572939 < \alpha < 1.574445.\]
\[ 1.148205 < \beta < 1.148230 \hspace{3mm} (\text{or }  1.318376 < \beta^2 < 1.31843).\]

The constant $\eta$ and the problem of counting maximal primitive subsets of the integers does not appear to have been considered before in the literature, so in this situation we have included the upper bound obtained by our method as well.

Despite our method's success at counting primitive subsets with a variety of properties, it isn't clear whether it can be used to count primitive subsets of sizes other than $\left\lceil\frac{n}{2}\right\rceil$. 
It would be interesting, for example, to study the distribution of the sizes of primitive subsets of $\{1,2 \ldots, n\}$.  As a partial result in this direction we use bounds for $\alpha$ and $\eta$ from Theorem \ref{thm:numalpha} to bound the median size of such a subset. 

\begin{theorem} \label{thm:median}
Let $\nu(n)$ denote the median of the sizes of the primitive subsets of $\{1,2,...n\}$.  Then for sufficiently large $n$, we have \[0.168153n <\nu(n) < 0.417739n.\]
\end{theorem}

It is natural to conjecture from this that $\nu(n) \sim vn$ for some constant $v$, this and other questions are posed in Section \ref{sec:questions}.

We briefly introduce a few additional applications of the main theorem to problems outside of primitive sets.  In Section \ref{sec:pathcov} we consider the problem of covering the divisor graph of $\{1,2,\ldots, n\}$ with as few vertex-disjoint paths as possible. A partition of the vertices into the least number of connected paths is known as a minimal path cover of the graph.) Various authors have studied the paths in this divisor graph.  Pomerance \cite{PDG} showed that the length of the longest vertex-disjoint path in this divisor graph has length $o(n)$ and now, due to Erd\H{o}s and Saias \cite{erdsai}, we know that the length of this longest path can be bounded above and below by positive constants multiplied by $\frac{n}{\log n}$.

Let $C(n)$ denote the minimal number of vertex disjoint paths required to cover the integers up to $n$ (sequence \href{http://oeis.org/A320536}{A320536} in the OEIS).  For example $C(7)=2$, as the divisor graph can be covered by the two paths $\{7,1,5\}$ and $\{3,6,2,4\}$ but it is not possible to include all the vertices in a single path.   This problem has been considered in several papers \cite{SaiasMelotti} after being introduced by Erd\H{o}s and Saias \cite{erdsai}.  Saias \cite{saiasdiv} showed that $\frac{n}{6} \leq C(n) \leq \frac{n}{4} $ for sufficiently large $n$. Mazet \cite{mazet} improves this, showing that $C(n) \sim cn$ for some constant $c$ satisfying $0.1706\leq c \leq 0.2289$, and Chadozeau \cite{chad} gives the error term $C(n) = cn\left(1+O\left(\frac{1}{\log \log n \log \log \log n}\right)\right)$.  Our main theorem allows us to improve this error term dramatically and also to improve both the lower and upper bounds for $c$.

\begin{theorem} \label{thm:pathcover}
The minimum number of paths required to cover the divisor graph of the integers up to $n$ satisfies \[C(n) = cn\left(1+O_\epsilon\left(\exp\left(-(1-\epsilon)\sqrt{\log n \log \log n}\right)\right)\right).\]  The constant $c$ is effectively computable, and satisfies the improved bounds \[0.190913<c<0.217838.\] 
\end{theorem}

In Section \ref{sec:gpf} we consider geometric-progression-free subsets of the integers up to $n$.   Here we consider geometric progressions of the form $a,ar,ar^2$ with a positive integer $a$, and an integer ratio $r$ greater than 1.  Unlike the case of primitive sets, where it is easy to see that the largest size a primitive subset of $\{1,2,\ldots, n\}$ can have is $\left\lceil\frac{n}{2}\right\rceil$, the problem of determining the maximal size of a subset of these integers avoiding 3-term geometric progressions is not so clear.  Various authors have considered the problem of determining the greatest possible density of a set of integers that is free of 3-term geometric progressions \cites{rankin,riddell,bbhs,NO,mcnewgpf}.  

Let $G(n)$ denote the size of the largest subset of the integers up to $n$ avoiding a 3-term geometric progression with integral ratio (sequence \href{http://oeis.org/A230490}{A230490} in the OEIS).  The argument for the bounds for the upper density of a subset of integers avoiding such progressions in \cite{mcnewgpf} shows that $G(n)\sim bn$ for an effectively computable constant $b$ satisfying \[0.81841 < b< 0.81922.\] 
The main theorem allows us to reprove this result with an error term.

\begin{theorem} \label{thm:gpf}
For any $\epsilon>0$, $G(n)$ the size of the largest geometric progression free subset of the integers up to $n$ satisfies \[G(n) = bn\left(1+O_\epsilon                                 \left(\exp\left(-(1-\epsilon)\sqrt{\log n \log \log n}\right)\right)\right).\]
\end{theorem}

We can likewise count the geometric progression free subsets of these integers, which we denote by $H(n)$.

\begin{theorem} \label{thm:theta}
The limit $\lim_{n \to \infty} H(n)^{1/n} = \theta$ exists.  Furthermore, for any $\epsilon>0$ and $n$ sufficiently large, 
\[H(n) = \theta^{n\left(1+O_\epsilon\left(\exp\left(-(1-\epsilon)\sqrt{\log n \log \log n}\right)\right)\right)}.\] The constant $\theta$ is effectively computable and satisfies the bounds
\[1.901448 < \theta < 1.925556.\]
\end{theorem}

Finally, in Section \ref{sec:proof} we conclude with several lemmas from analytic number theory which lead to the proof of the main theorem.

\section{Counting Primitive Sets} \label{sec:countprim}

We first show how the main theorem, Theorem \ref{thm:main}, can be used to estimate $Q(n)$, the count of the primitive subsets of the integers up to $n$.  We fix $n$, and count the possible primitive subsets of $\{1,2,\ldots, n\}$ by working backward from the end.  Recall that any subset of the integers in $\left(\frac{n}{2},n\right]$ is primitive.  Each element in this range can either be included or not so $Q(n) \geq 2^{n/2}$.  We can make this bigger by considering the integers in $\left(\frac{n}{3},\frac{n}{2}\right]$.   For each such $k$ in this range we can either include $k, 2k$ or neither.   Thus for each such integer there are 3 possibilities, replacing the two possibilities when only $2k$ was considered.  Thus $\displaystyle{Q(n)\geq 2^{n/2}\left(\tfrac{3}{2}\right)^{n/6} = 2^{n/3}3^{n/6} \approx 1.5131^n}$.  

We generalize this to any $k\leq n$ by defining \begin{equation}
    r(k,n) = \frac{\# \text{Primitive subsets of $[k,n]$}}{\# \text{Primitive subsets of $[k+1,n]$}} \label{eqn:rkn}
\end{equation}
to be the contribution of $k$ to our primitive set count working backward from the end.  With this definition we see that the product of the $r(k,n)$ telescopes leaving
\[Q(n) = \prod_{k=1}^n r(k,n) =\frac{\# \text{Primitive subsets of [1,$n$]}}{\# \text{Primitive subsets of }\varnothing }.\]

We can apply Theorem \ref{thm:main} by taking $f(k,n)=\log(r(k,n))$. To do so, it is necessary to see that $f(k,n)$ depends only on the connected component of $k$ in the divisibility graph of $[k,n]$.  This is easily seen however, as the inclusion of $k$ (or any integer in the interval $[k,n]$) in a primitive set depends only on whether its multiples (or divisors) are already included.  The count of the number of primitive subsets of an interval can therefore be computed as the product of the number of primitive subsets (or independent sets of vertices when viewed as a graph) of each connected component of the divisibility graph.  The contribution from each component of the divisibility graph that doesn't contain $k$ will cancel in the numerator and denominator of \eqref{eqn:rkn}.  This leaves only the ratio of the number of primitive (independent) subsets of the connected component of $k$ to the number of primitive (independent) subsets of this component with the element $k$ removed.  As this is purely a graph theoretic question, it is clear that the answer will be the same for any other isomorphic divisibility graph. 

Note also that $1\leq r(k,n) \leq 2$, since every primitive subset of $[k+1,n]$ is also a primitive subset of $[k,n]$, and the number of primitive sets could, at most, double after $k$ is considered if it is possible to append $k$ to every primitive subset of $[k+1,n]$ and still get a primitive set.  Thus we have $0 \leq f(k,n) \leq \log 2$. 

Applying Theorem \ref{thm:main} we have 
\begin{align*}
    Q(n) & = \prod_{i=1}^n r(i,n) = \exp\left(\sum_{i=1}^n f(i,n)\right) \\
    &= \exp\left(n\log \alpha +  O\left(n\exp\left(-(1-\epsilon)\sqrt{\log n \log \log n}\right)\right)\right)
\end{align*}
with \begin{align} \alpha &= \exp\left(\sum_{i=1}^\infty  \sum_{\substack{d\\ P^+(d)\leq i}} \sum_{t \in [id,(i+1)d)} \left(\frac{f(d,t)}{t(t{+}1)}\prod_{p\leq i} \frac{p{-}1}{p}\right)\right) \nonumber \\
&= \prod_{i=1}^\infty  \prod_{\substack{d\\ P^+(d)\leq i}} \prod_{t \in [id,(i+1)d)} \exp \left(\frac{f(d,t)}{t(t{+}1)}\prod_{p\leq i} \frac{p{-}1}{p}\right) \nonumber \\
&= \prod_{i=1}^\infty  \prod_{\substack{d\\ P^+(d)\leq i}} \prod_{t \in [id,(i+1)d)} r(d,t)^{\frac{1}{t(t{+}1)}\prod_{p\leq i} \frac{p{-}1}{p}}. \label{eqn:alpha} \end{align}
This proves Theorem \ref{thm:alpha}.

The proofs of Theorems \ref{thm:beta} and \ref{thm:eta} are very similar.  We define 
\begin{equation}
    s(k,n) = \frac{\# \text{Primitive subsets of $[k,n]$ of maximum size}}{\# \text{Primitive subsets of $[k+1,n]$ of maximum size}} \label{eqn:skn}
\end{equation}
and 
\begin{equation}
    w(k,n) = \frac{\# \text{Maximal primitive subsets of $[k,n]$}}{\# \text{Maximal primitive subsets of $[k+1,n]$}}. \label{eqn:tkn}
\end{equation}
It follows that these functions depend only on the connected component of $k$ in the divisibility graph of $[k,n]$ by the same argument as for $r(k,n)$.  Unlike the case of $r(k,n)$, however, some care is required to show that these functions are bounded so that Theorem \ref{thm:main} can be applied. We will show that both $1\leq s(k,n) \leq 2$ and $1\leq w(k,n) \leq 2$.

First note that if $k>\frac{n}{2}$ then $s(k,n)=w(k,n)=1$, since the only maximum or maximal subsets of $[k,n]$ is the entire set. Thus the numerators and denominators of each term of \eqref{eqn:skn} and \eqref{eqn:tkn} are all 1. (Alternatively, note that $k$ is the only vertex in the component connected to $k$ in the divisor graph of $[k,n]$, and so the element $k$ must be included.) Thus we can restrict our attention to the case $k\leq \frac{n}{2}$.

In this case, we find that any primitive subset of either $[k,n]$ or $[k+1,n]$ of maximal size has size $n/2$, and thus every maximum primitive subset of $[k+1,n]$ is also a maximum primitive subset of $[k,n]$, so $s(k,n)$ is at least one.  For $w(k,n)$ we show that every maximal primitive subset of $[k+1,n]$ is also a maximal primitive subset of $[k,n]$. Since the only neighbors of $2k$ in the divisor graph of $[k+1,n]$ are its multiples, (every divisor of $2k$ is less than $k+1$) either $2k$ or one of its multiples must be included in any maximal primitive subset of $[k+1,n]$.  Every multiple of $2k$ is also a multiple of $k$, so the set is also a maximal primitive subset of $[k,n]$.   Thus we have also that $w(k,n)$ is always at least 1.

To get upper bounds for these functions, we produce an injection from maximum (respectively maximal) primitive subsets of $[k,n]$ containing $k$ to maximum (maximal) primitive subsets of $[k+1,n]$. First note that any maximum primitive subset of $[k,n]$ that contains $k$ can be put into correspondence with a primitive subset of $[k+1,n]$ of the same size by replacing $k$ by $2k$. As every neighbor of $2k$ in the divisor graph is also a neighbor of $k$, the set will remain primitive.  This mapping is clearly injective since the step of replacing $k$ by $2k$ is easily reversed. Thus $s(k,n)\leq 2$.  For maximal primitive sets, simply replacing $k$ by $2k$ might not produce a maximal primitive subset if it is possible also to add other (odd) multiples of $k$ once $k$ is removed.  We therefore remove $k$ and instead include $pk$ for each prime number $p$ with $pk\leq n$ that can be added to the set while still producing a primitive set.  Again, it is easy to reconstruct the original set by removing the multiples of $k$ so it is clear that this again is an injection and $w(k,n)\leq 2$ as well.  

We then set 
\[M(n) = \prod_{k=1}^n s(k,n), 
\hspace{2cm} m(n) = \prod_{k=1}^n w(k,n)\]
which each telescope in the same way as the product for $Q(n)$, and apply Theorem \ref{thm:main}, proving Theorems \ref{thm:beta} and \ref{thm:eta} with 
\begin{equation} \beta = \prod_{i=1}^\infty  \prod_{\substack{d\\ P^+(d)\leq i}} \prod_{t \in [id,(i+1)d)} s(d,t)^{\frac{1}{t(t{+}1)}\prod_{p\leq i} \frac{p{-}1}{p}} \label{eqn:beta} \end{equation}

and

\begin{equation} \eta = \prod_{i=1}^\infty  \prod_{\substack{d\\ P^+(d)\leq i}} \prod_{t \in [id,(i+1)d)} w(d,t)^{\frac{1}{t(t{+}1)}\prod_{p\leq i} \frac{p{-}1}{p}}. \label{eqn:eta} \end{equation}

\section{Numerical Estimates} \label{sec:numerics}
Using the product formulas for $\alpha$, $\beta$ and $\eta$, given in \eqref{eqn:alpha}, \eqref{eqn:beta} and \eqref{eqn:eta} respectively, it is possible to compute each constant to arbitrary precision by computing sufficiently many terms in the product.  We know that the complete sum \[\sum_{i=1}^\infty  \sum_{\substack{d\\ P^+(d)\leq i}} \sum_{t \in [id,(i+1)d)} \frac{1}{t(t+1)}\prod_{p\leq i} \frac{p{-}1}{p} \ = \ 1. \]
Suppose we have lower and upper bounds $L\leq f(d,t) \leq U$ for all pairs $(d,t)$, and we compute the values of $f(d,t)$ for some subset $S \subset \{(i,d) \mid P^+(d)\leq i\}$ and all $t \in [id,(i+1)d)$.  Then we are able to numerically bound $C_f$ by \[L+\sum_{(i,d) \in S} \sum_{t \in [id,(i+1)d)} \frac{f(d,t)-L}{t(t+1)}\prod_{p\leq i} \frac{p-1}{p} \ \leq \ C_f \]
and
\[ C_f \ \leq \ U-\sum_{(i,d) \in S} \sum_{t \in [id,(i+1)d)} \frac{B-f(d,t)}{t(t+1)}\prod_{p\leq i} \frac{p-1}{p}.\]

The constants were estimated using computations done using SageMath \cite{sagemath} and code written in C++. The following observations were used during the computation to improve the speed of convergence.  
\begin{observation} \label{obs:consecutivet}
For any function $f(d,t)$ depending only on the connected component of $d$ in the the divisor graph of $[d,t]$ and fixed values of $i$ and $d$, if $t+1 \in (id,(i+1)d)$ is not an $i$-smooth number then $f(d,t)=f(d,t+1)$.
\end{observation}
This follows after noting for any prime $p>t/d$ that if $p^\nu$ divides some element in the connected component of $d$ in the divisor graph of $[d,t]$ then it divides every element of this component.  (See Lemma \ref{lem:isodivgra}.)  Because the $t+1$ is not $i$-smooth, it must be divisible by a prime $p>i$ that does not divide $d$, and hence is not part of the connected component of $d$.  Thus for each pair $i$, $d$ it is only necessary to compute $f(d,t)$ for the $i$-smooth values of $t$ in the interval $[id,(i+1)d)$.

\begin{observation} \label{obs:primes}
Suppose $f(d,t)$ depends only on the connected component of $d$ in the the divisor graph of $[d,t]$ and fix values $i,d$ and $t \in [id,(i+1)d)$. If, for some $p\leq i$, every term in the connected component of $pd$ in the divisibility graph of $[pd,pt]$ is divisible by $p$, then $f(pd,pt)=f(d,t)$.  Furthermore,  $f(p^jd,p^jt)=f(d,t)$ for all $j>0$.
\end{observation}
\begin{proof}
If every term in the connected component of $pd$ in the divisibility graph of this interval is divisible by $p$, then it is clear that this component is isomorphic to the component obtained by dividing everything by $p$.  

Now, suppose every term in this component is divisible by $p$, and consider the interval $[p^2d,p^2t]$.  Clearly the component connected to $p^2d$ contains every term in the component of the divisibility graph of $[d,t]$ connected to $d$ multiplied by $p^2$.  Now suppose for contradiction that the connected component of $dp^2$ contained an additional term $a>dp^2$ not divisible by $p^2$.  Then this term must be a divisor of some term $bp^2$ in the connected component of $dp^2$ where $p\nmid b$ (or if not, it could be replaced by one that is, since there must exist a path in this graph connecting $a$ to $dp^2$).

This means the ratio between $a$ and $bp^2$ is divisible by $p$, and so \[bp^2/p = bp \geq a >dp^2.\]  Thus we have $b>dp$ and so $b$ would be an element of the connected component of $dp$ in the divisor graph of $[pd,pt]$ which is a contradiction. 
\end{proof}

One consequence of Observation \ref{obs:primes} is that the sum (or product) over $d$ becomes finite for values of $i\leq 4$.  For such values of $i$, any value of $d>48$ will result in a connected graph component in which every term shares a common prime factor. (In fact, one can check that the connected component of $1$ in the divisor graph of $\mathbb{Q}\cap[1,5)$ consists of the numbers $\{1,\frac{9}{8},\frac{4}{3},\frac{3}{2},\frac{27}{16},2,\frac{9}{4},\frac{8}{3},3,\frac{27}{8},4,\frac{9}{2}\}$ since only multiplications or divisions by 2 or 3 are allowed.)  Thus the challenge is to numerically estimate the terms for $i\geq 5$.

Determining the optimal order in which to consider the pairs of $i$ and $d$ in the infinite product is challenging because $r(d,t)$ (as well as the other functions considered in this paper) becomes more difficult to compute both for larger values of $d$ and $t$.  At the same time, the values of $r(d,t)$ tend toward one rapidly as $i=\lfloor\frac{t}{d}\rfloor$ gets large.  Experimentally, the best lower bounds obtained by the author were obtained by considering potential pairs $i$,$d$ in increasing order of the value of the expression $d\times i^5$.

By computing $r(d,t)$ for all pairs of $i,d$ with $d$ being an $i$-smooth integer satisfying $di^5\leq 10^8$, (as well as the extended ranges $d<11250000$ when $i=5$, $d<2400000$ for $i=6$ and $d<27440$ when $i=7$) and all $i$-smooth values of $t\in [id,(i+1)d)$, and using these terms in the product \eqref{eqn:alpha} along with the additional terms obtained by applying Observation \ref{obs:primes}, the lower bound  is obtained.  

In a similar fashion one can compute initial terms for $\beta$ and $\eta$.  The computation for $\beta$ involved computing $s(d,t)$ for all values of $d$ and $i$ satisfying $di^5<1.6\times10^9$ as well as those satisfying $di^5<3.2\times10^9$ for $i\leq 10$. For $\eta$ all values of $d$ and $i$ satisfying $di^5<3.2\times 10^7$ were taken into account, as well as those satisfying $di^5<6.4\times10^7$ for $i \leq 7$.  Again, in both computations Observations \ref{obs:consecutivet} and \ref{obs:primes} were used to speed up the computation and take into account additional pairs $d,t$.

\section{Conjectural Improved Upper Bounds} \label{sec:conj}

We give here a different method of counting primitive sets which would give a much better upper bound for the constant $\alpha$ if Conjecture \ref{conj:gmul} were proven, and may be of some independent interest.

Here we count primitive subsets of $\{1,2,\ldots, n\}$ by working forward from 1, rather than backward from $n$ as was done in the previous section.  We define $g(1)=1$ and for $k>1$ \begin{align*}
    g(k) &= \frac{Q(k)}{Q(k-1)}-1 = \frac{\# \text{Primitive subsets of $[1,k]$ that include $k$}}{\# \text{Primitive subsets of $[1,k-1]$}}.
\end{align*}

The reason we defined $g(k)$ this way, as the ratio of primitive sets up to $k$ that include the integer $k$ to the primitive sets up to $k$ that do not, is that the resulting function appears to be submultiplicative in the following sense.

\begin{conjecture} \label{conj:gmul}
    The function $g(n)$ is submultiplicative.   If (n,m)=1 then \[g(nm)\leq g(n)g(m).\] 
    Furthermore, for any prime $p$, $g(p^{i+1})\leq g(p^{i})$.
\end{conjecture}

The conjecture has been verified for all values of $n<899$.  

If one assumes the conjecture, then a substantially improved upper bound for $\alpha$ can be computed. Using the computed values of $g(n)$, $n<899$, and taking into account all 31-smooth numbers gives the upper bound $\alpha <1.573487$.

\section{The median size of a primitive subset}
    
The computations obtained in the previous section will also allow us to give upper and lower bounds for $\nu(n)$, the median size of a primitive subset of the integers up to $n$.  First we need a lemma (see for example \cite[solved exercise 9.4]{knuth}) on sums of binomial coefficients, which, given a set of size $n$, counts the number of subsets of size at most $\lambda n$.
    \begin{lemma} \label{lem:sumbin}
    Fix $0 <\lambda < \frac{1}{2}$.  The partial sum of binomial coefficients satisfies
    \[\sum_{i=0}^{\lfloor \lambda n \rfloor} \binom{n}{i} = 2^{nh(\lambda)+O(\log n)}\]
    where $h(\lambda)=-\lambda\log_2 \lambda +(\lambda-1)\log_2(1-\lambda)$ and $\log_2$ is the logarithm base 2.
    \end{lemma}

Using this, and the numerical computations of $\alpha$ and $\beta$, we can bound the median size of a primitive subset of $\{1,2,\ldots, n\}$.

\begin{proof}[Proof of Theorem \ref{thm:median}]
Let $\nu(n)$ denote the median size of the primitive subsets of the integers up to $n$.  Since half of the primitive subsets of this set must have size at most $\nu(n)$, and the total number of primitive subsets is approximately $\alpha^n$, we can obtain a lower bound for $\nu(n)$ by supposing that all small subsets of the integers up to $n$ are primitive.  Combining this with Lemma \ref{lem:sumbin} we get that  

\[\sum_{i=0}^{\nu(n)} \binom{n}{i} = 2^{nh(\nu(n)/n)+O(\log n)} \geq \frac{1}{2}Q(n) =  \alpha^{n\left(1+O_\epsilon\left(\exp\left(-(1-\epsilon)\sqrt{\log n \log \log n}\right)\right)\right)}. \]

Taking logs gives \[h(\nu(n)/n) \geq \frac{\log \alpha }{\log 2} +O_\epsilon\left(\exp\left(-(1-\epsilon)\sqrt{\log n \log \log n}\right)\right)\] or \[\nu(n) \geq nH^{-1}\left(\frac{\log \alpha }{\log 2}\right)(1+o(1)).\]

Using the lower bound for $\alpha$ gives $\nu(n) \geq 0.168153n$ for sufficiently large $n$. To get an upper bound, we must use in addition the lower bound obtained for the constant $\eta$.  Every primitive set is contained in a maximal primitive set (but not necessarily a maximum one).  We count those primitive sets obtained by starting with a maximal primitive set $S$ and removing some subset $T$ from it.  Since half of the primitive subsets up to $n$ have size at least $\nu(n)$, and every maximal primitive set has size at most $\frac{n}{2}$, it must be possible to construct at least half of the primitive subsets up to $n$ by starting with a maximal primitive set and removing at most $\frac{n}{2}-\nu(n)$ terms from it.  Thus the median size of a primitive subset must be sufficiently less than $\frac{n}{2}$ to produce $\frac{1}{2}Q(n)$ sets in this way.  In particular we have the inequality \[m(n)\sum_{i=0}^{\frac{n}{2}-\nu(n)} \binom{n/2}{i}\geq \frac{1}{2}Q(n).\]
Using Lemma \ref{lem:sumbin} and Theorems \ref{thm:alpha} and \ref{thm:eta}, and taking logs, gives 
\[h\left(1-\frac{2\nu(n)}{n}\right) \geq \frac{2(\log \alpha-\log \eta)}{\log 2} +O_\epsilon\left(\exp\left(-(1-\epsilon)\sqrt{\log n \log \log n}\right)\right)\]
so that $\nu(n)\leq \frac{n}{2}\left(1-h^{-1}\left(\frac{2(\log \alpha-\log \eta)}{\log 2}\right)+o(1)\right)$.  Using the lower bound for $\alpha$, and the upper bound for $\eta$, we get $\nu(n) \leq 0.417739n$ for sufficiently large $n$.
\end{proof}

\section{An optimal path covering of the divisor graph} \label{sec:pathcov}
As mentioned in the introduction, it is known due to Mazet \cite{mazet} and Chadozeau \cite{chad} that $C(n)$, the minimum number of disjoint paths necessary to cover the divisor graph of the integers up to $n$ satisfies
\[C(n)=cn\left(1 +O\left(\frac{1}{\log \log n \log \log \log n}\right)\right).\]
for a constant $0.1706\leq c\leq 0.2289$.
Using our main theorem we show (Theorem \ref{thm:pathcover})
 \[C(n) = cn\left(1+O_\epsilon\left(\exp\left(-(1-\epsilon)\sqrt{\log n \log \log n}\right)\right)\right)\]  and $0.176448<c$. 

\begin{proof}[Proof of Theorem \ref{thm:pathcover}]
As in previous examples, we set up a telescoping sum, however in this case there is no need to take logs.  Define  \[V(k,n)=\#\{\text{disjoint paths needed to cover the divisor graph of }[k,n]\} \] and define $v(k,n)=V(k,n)-V(k+1,n)$ so that \[C(n)=V(1,n) = \sum_{k=1}^n v(k,n).\]  Note that $v(k,n) \in \{-1,0,1\}$, since in the worst case scenario, we can cover the additional vertex $k$ using a single new path, and in the best case scenario, we are able to join paths from two different components of the divisor graph of $[k+1,n]$ through the vertex $k$.  Since $k$ can only be part of one path however, it is never possible to decrease the number of paths required by more than 1.  

It is also clear that $v(k,n)$ depends only on the connected component of $k$ in the divisor graph of $[k,n]$, so we can apply the main theorem to $C(n)$, which completes the proof, with 
\[c= \sum_{i=1}^\infty  \sum_{\substack{d\\ P^+(d)\leq i}} \sum_{t \in [id,(i+1)d)} \left(\frac{v(d,t)}{t(t{+}1)}\prod_{p\leq i} \frac{p{-}1}{p}\right). \]
Note that $v(1,1)=1$, but is otherwise nonpositive for any of the terms included in the sum above.  (If $i\geq 2$, then when computing $v(d,t)$, the element $2d$ is a vertex of the graph, and any neighbor of $2d$ is also a neighbor of $d$.  Thus it is always possible to cover the vertex $d$ by including it in the same path that is used to cover $2d$.)

For small values of $d$ and $t$ it seems that $v(d,t)$ is generally equal to $-1$ unless $d$ is divisible by $6$ in which case it is generally 0.  This pattern becomes gradually less pronounced as the values of $d$ and $t$ become larger.  Values of $v(d,t)$ were computed in the range $id^5<4\times 10^8$ as well as the extended range $id^5<8\times 10^8$ for $d\leq 825$. Along the way a small number of pairs were skipped when the computation time of the value of $f(d,t)$ exceeded a predefined timeout. Using the bounds $-1 \leq v(d,t)\leq 0$ for all of the values that were not computed we obtained the bounds $0.190913<c<0.217838$.
\end{proof}

\section{Geometric Progression Free Sets} \label{sec:gpf}  
Geometric progression free sets are similar to primitive sets, in fact one can characterize primitive sets as those avoiding geometric progressions of length 2. So it is not surprising that the methods similar to those used to study primitive sets apply to this situation as well.  In this situation it isn't a divisibility relation between two integers that matters, but rather a relationship among three integers $a,ar$ and $ar^2$.  In place of the divisor graph, we instead consider a geometric-progression hypergraph in which 3 vertices are connected by a hyperedge if the corresponding vertices are in geometric progression.

Unlike primitive sets, we don't know precisely the maximum size $G(n)$ of a subset of $\{1,2,\ldots, n\}$ avoiding three term geometric progressions with integral ratio.  It is known that $G(n) \sim bn$ for some constant $b$ with $0.81841<b<0.81922$, and that $b$ is effectively computable.  While we don't improve these bounds using the current method, we do obtain an improved error term  $G(n) = bn\left(1+O_\epsilon\left(\exp\left(-(1-\epsilon)\sqrt{\log n \log \log n}\right)\right)\right)$ for any $\epsilon>0$ in Theorem \ref{thm:gpf}.
\begin{proof}
While the underlying structure in this case is a geometric-progression hypergraph, rather than the divisor graph, one can easily check that the proof of Lemma \ref{lem:isodivgra} applies just as well when the divisor graph is replaced by a geometric-progression hypergraph.\footnote{In fact,in Lemma \ref{lem:isodivgra},  for a geometric-progression hypergraph, we can take $i=\left\lfloor\sqrt{\tfrac{n}{a}}\right\rfloor$ instead of $\left\lfloor\frac{n}{a}\right\rfloor$.  Using this instead in the proof of the main theorem, taking intervals of the form $\left(\frac{n}{(i+1)^2},\frac{n}{i^2}\right]$ but otherwise proceeding the same, we can get a small further improvement to the error term.  The error term $E_1$ is the limiting factor and the resulting final error term is $O_\epsilon\left(An\exp\left(-\left(\frac{3}{2\sqrt{2}}-\epsilon\right)\sqrt{\log n \log \log n}\right)\right)$.}  The proof of the main theorem then applies just as well to functions which depend only on the geometric-progression hypergraph of the interval $[d,t]$.

Now, we let $G(d,t)$ denote the size of a geometric progression free subset of the integers $[d,t]$ of maximal size, and let $g(d,t) = G(d,t)-G(d+1,t)$, so that \[G(n) = G(1,n) = \sum_{k\leq n} g(k,n).\]
Now $g(d,t)$ depends only on the geometric-progression hypergraph of the interval $[d,t]$, and only ever takes the values 0 or 1, and the result follows.
\end{proof}

We can likewise use the main theorem applied to geometric-progression hypergraphs to count (Theorem \ref{thm:theta}) the number $H(n)$ of subsets of the integers up to $n$ that avoid 3-term-geometric progressions with integral ratio.  The proof is nearly identical to that of Theorem \ref{thm:alpha}.  To obtain the bounds on the constant $\theta$, we find that the analogous version of Observation \ref{obs:primes} for geometric progressions allows us to take into account the contribution from all values of $d$ for $1\leq i \leq 11$, as well as the contribution from those values of $d$ up to those listed in Table \ref{tab:gpf} in order to get the bounds $1.901448 < \theta < 1.925556$.

\section{Questions} \label{sec:questions}

Despite the wide range of problems that can be tackled using the main theorem, this work leaves several open questions that will require new techniques.  In addition to the conjecture described in Section \ref{sec:conj}, we pose two questions that seem interesting enough to study further.

\begin{question}
Is the median size $\nu(n)$ of the primitive subsets of the integers up to $n$ asymptotic to $vn$ for some constant $v$? If so, is there an algorithm to compute $v$ to arbitrary precision?
\end{question}

In Section \ref{sec:gpf} we considered sets that avoided 3-term geometric progressions with integral ratio.  The results in \cite{mcnewgpf} imply that the largest subset of the integers up to $n$ avoiding 3-term geometric progressions with \textit{rational} ratio is also asymptotic to a different effectively computable constant times $n$.
\begin{question}
Is it possible to give an error term analogous to Theorem \ref{thm:gpf} for the size of a subset of the integers up to $n$ of maximal size avoiding 3-term geometric progressions with rational ratio? Can we count the number of such subsets analogously to Theorem \ref{thm:theta}?
\end{question}

\section{Proof of the Main Theorem} \label{sec:proof}

It remains to prove the main theorem upon which the results of this paper rely.  First we need some notation and lemmas from analytic number theory. Denote by $\Psi(x,y)=\#\{n\leq x\mid P^+(n)\leq y\}$ the number of $y$-smooth integers up to $x$.  We will frequently use the upper bound for smooth numbers, 
\[\Psi(x,y) \ll x \exp\left((-1+o(1))\frac{\log x}{\log y}\log \frac{\log x}{\log y}\right) \]
valid for $y\geq \left(\log x\right)^{1+\epsilon}$, as well as the uniform bound 
\[\Psi(x,y) \ll_\epsilon x \exp\left(\frac{\log x}{\log y}\log \frac{\log x}{\log y}\right) +x^\epsilon\] valid for all $x,y \geq 2$ and $\epsilon>0$ (see \cite[Section III.5]{IAPNT}).  On the other hand a number $n$ is called $y$-rough if $P^-(n)>y$.

We will need the following estimates of Tenenbaum \cite{Tenenbaum-LFC} for the function $\Theta(x,y,z)$ which counts integers $n\leq x$ with a $y$-smooth divisor  $d>z$, as well as the related function $$S(y,z)= \sum_{\substack{d>z\\P^+(d)\leq y}} \frac{1}{d}.$$
    \begin{lemma} \label{lem:recipsmooth} For all $\epsilon>0$, $x, y\geq 2$ and $z<\exp \exp (\log y)^{(3/5-\epsilon)}$
    \[\Theta(x,y,z) = (1+o(1))xS(y,z)\prod_{p<y}\frac{p-1}{p} \ll x\exp\left((-1+O(1))\frac{\log z}{\log y}\log \frac{\log z}{\log y}\right).\]
\end{lemma}

We will also need to bound the number of $y$-rough integers in an interval.
    
    \begin{lemma} \label{lem:sieveinterval}
    Let $\epsilon>0$, $I$ an interval of length $X>y>2$ and define $u=\frac{\log X}{\log y}$. The number of integers in $I$ free of prime divisors up to $y$ is \[X\prod_{p\leq y}\frac{p-1}{p}+O\left(X\exp\left((-1+o(1))u\log u\right)\right)  +O_\epsilon(X^\epsilon)\]
    as $u \to \infty$.
    
    \end{lemma}

    \begin{proof}
    This follows from the ``Fundamental Lemma'' of Brun's sieve, see for example Theorem 6.12 of \cite{FI}. Taking the level of distribution to be $X$, the number of such integers in this interval is \[X\prod_{p\leq y}\frac{p-1}{p}\left(1+O_\epsilon\left(\exp\left((-1+o(1))u\log u\right)\right) \right) +\sum_{\substack{d\leq X\\P^{+}(d)\leq y}}O(1).\]  Now we can bound the second error term above by $\Psi(X,y) \ll_\epsilon Xu^{-u}+X^\epsilon$  and the result follows by approximating $\prod_{p\leq y} \frac{p-1}{p} = O\left(\frac{1}{\log y}\right)$.
    \end{proof}
    
    Finally, we will use the following lemma to characterize the connected component of the divisor graph of an interval.
    
        \begin{lemma} \label{lem:isodivgra}
    Fix a positive integer $n$, suppose that $0<a\leq n$, and set $i= \left\lfloor \frac{n}{a}\right\rfloor$.  Let $d|a$ be the largest $i$-smooth divisor of $a$, and $\ell =\frac{a}{d}$ the ``$i$-rough'' part of $a$.  Finally let $t=\left\lfloor \frac{n}{\ell}\right\rfloor$.  Then the connected component of $a$ in the divisor graph of $[a,n]$ is isomorphic to the connected component of $d$ in the divisor graph of $[d,t]$ (with the vertex of $a$ corresponding to the vertex of $d$). 
    \end{lemma}
    \begin{proof}  Suppose that $a\leq b<c\leq n$ are any two connected vertices in the connected component of $a$ in this divisor graph, and let $r=\frac{c}{b}$ be the (necessarily integral) ratio between them.  Since $r\leq \frac{n}{a}<i+1$, the ratio $r$ cannot have any prime factors greater than $i$.  So $b$ and $c$ are divisible by all of the same prime factors greater than $i$ to the same powers.  
    
    Recall that $\ell$ is divisible only by primes greater than $i$ and $\ell|a$ so $\ell$ divides all of the integers in the connected component of $a$.  We defined $t=\left\lfloor \frac{n}{\ell}\right\rfloor$ so that $t\ell $ is the largest integer less than or equal to $n$ divisible by $\ell$, and thus the largest number from this interval that could possibly be part of the component of the divisor graph connected to $a$.  So the connected component of $a$ in $[a,n]$ is the same as the connected component of $a$ in $[a,t\ell]$.
    
    Now we can divide each integer in this connected component by $\ell$, and see that the connected component of $a$ in $[a,t\ell]$ is the same as the connected component of $\left[\frac{a}{\ell},t\right]=[d,t]$ with the isomorphism being multiplication by $\ell$. 
    \end{proof}

    \begin{proof}[Proof of main theorem]
    Fix $\epsilon>0$, $A\geq 0$ and suppose $f(k,n)$ is bounded in absolute value by $A$ and depends only on the connected component of $k$ in the divisor graph of $[k,n]$. Our goal is to estimate $\sum_{k=1}^n f(k,n)$ as $nC_f$ where 
    \begin{equation}C_f = \sum_{i=1}^\infty  \sum_{\substack{d\\ P^+(d)\leq i}} \sum_{t \in [id,(i+1)d)} \left(\frac{f(d,t)}{t(t+1)}\prod_{p\leq i} \frac{p-1}{p}\right). \label{eq:cf} \end{equation}

Our goal is to use Lemma \ref{lem:isodivgra} to group together the equal terms in this sum.  First, we truncate the sum, removing those $k\leq \frac{n}{N+1}$ for some parameter $N$ to be chosen later.  This allows us to group together those $k$ having the same value of $i=\left\lfloor \frac{n}{k} \right\rfloor$ for $i\leq N$.

\begin{align*}
    \sum_{k=1}^n f(k,n) &= \sum_{k=\frac{n}{N+1}}^n f(k,n) +O\left(\frac{An}{N}\right) \\
    &= \sum_{i=1}^N\sum_{k\in(\frac{n}{i+1},\frac{n}{i}]} f(k,n) +O\left(\frac{An}{N}\right).
\end{align*}
We now omit those values of $k$ whose largest $i$-smooth divisor is greater than some parameter $M$ to be determined.  In doing so, we introduce an error term $E_1$ to account for the omitted terms.  We then group together those values of $k$ in the sum above having the same largest $i$-smooth divisor $d$.

\begin{align}
     \sum_{k=1}^n f(k,n) &= \sum_{i=1}^N \left( \sum_{\substack{k\in(\frac{n}{i+1},\frac{n}{i}] \\  d|k, P^+(d)\leq i \Rightarrow d\leq M}} \hspace{-5mm}f(k,n) \right)+ E_1 +O\left(\frac{An}{N}\right) \nonumber \\
    &= \sum_{i=1}^N\left(\sum_{\substack{d\leq M \\ P^+(d) \leq i }} \sum_{\substack{k\in(\frac{n}{i+1},\frac{n}{i}]\\ d|k \\\  P^{-}\left(\frac{k}{d}\right) > i}} f(k,n) \right) + E_1 +O\left(\frac{An}{N}\right). \label{eq:truncsum}
\end{align}

The error term $E_1$ accounts for the contribution from those terms in each interval $(\frac{n}{i+1},\frac{n}{i}]$ having an $i$-smooth divisor greater than $M$, so  
\[E_1 \leq A\sum_{i=1}^N\left(\Theta\left(\tfrac{n}{i},i,M\right)-\Theta\left(\tfrac{n}{i+1},i,M\right)\right)\]  (recall $\Theta(\tfrac{n}{i},i,M)$ counts  integers up to $\tfrac{n}{i}$ with an $i$-smooth divisor exceeding $M$).

In the main term sum of \eqref{eq:truncsum} we group those values of $k$ with the same value of $t=\left\lfloor\frac{nd}{k}\right\rfloor$.  For fixed $i$ and $d$, the possible values of $t$ are thus each of the integers in the interval $[id,(i+1)d)$.  Doing so, this sum becomes
\begin{align}
        \sum_{i=1}^{N}  \sum_{\substack{d<M\\ P^+(d)\leq i}} \sum_{t \in [id,(i+1)d)} &
        \sum_{\substack{k \in \left(\frac{dn}{t+1},\frac{dn}{t}\right]\\ d\mid k \\ P^{-}\left(\frac{k}{d}\right)> i}} f(k,n) \nonumber \\ 
        &=\sum_{i=1}^{N } \sum_{\substack{d<M\\ P^+(d)\leq i}} \sum_{t \in [id,(i+1)d)}\left( f(d,t)
        \sum_{\substack{\ell \in \left(\frac{n}{(t+1)},\frac{n}{t}\right] \\ P^-(\ell)> i}} 1  \right). \label{eq:transformedsum}
\end{align}
Here we have obtained the second line above from the first using Lemma \ref{lem:isodivgra}.

The innermost sum above counts integers sifted of primes up to $i$ in the interval $\left(\frac{n}{(t+1)},\frac{n}{t}\right]$ which has size $\frac{n}{t(t+1)}$.  So, as long as both $N$ and $M$ are chosen not too large (meaning that $t<NM$ is not too large), these intervals are sufficiently long to apply Lemma \ref{lem:sieveinterval} to approximate the count of such integers.  Doing so gives  
        
    \begin{align}
        \sum_{\substack{\ell \in \left(\frac{n}{(t+1)},\frac{n}{t}\right] \\ P^-(\ell)> i}} 1 &= 
        \frac{n}{t(t{+}1)}\left(\prod_{p<i} \tfrac{p-1}{p}+O\left(\exp\left((-1{+}o(1))u_t \log u_t \right)\right)\right) +O_{A,\epsilon}\left(\frac{n^\epsilon}{t^{2\epsilon}}\right)\nonumber
    \end{align}
as long as $u_t = \frac{\log \frac{n}{t(t+1)}}{\log i}$, corresponding to the length of the interval in which Lemma  \ref{lem:sieveinterval} was applied, tends to infinity. Thus it will be necessary to choose $N, M$ such that \begin{equation}
            \log N =o\left(\log \frac{n}{M^2}\right) \label{ineq:utinfty}
        \end{equation}
        as $n\to \infty$.  Inserting this expression into \eqref{eq:transformedsum} and introducing another error term $E_2$ that expression becomes 
        \begin{align}
        \sum_{i=1}^{N }
        \sum_{\substack{d<M\\ P^+(d)\leq i}} \sum_{t \in [id,(i+1)d)}
        {\frac{f(d,t)n}{t(t+1)}\prod_{p<i} \frac{p-1}{p}} \ + \  E_2  \label{eqn:firstapprox}
        \end{align}
        where 
        \begin{align}
            E_2 &\ll_\epsilon A\sum_{i=1}^N \sum_{\substack{d<M\\ P^+(d)<i}} \sum_{t \in [id,(i+1)d)} \hspace{-2mm} \left(
        \frac{n}{t^2}\exp\left((-1{+}o(1))u_t \log u_t \right) + \frac{n^\epsilon}{t^{2\epsilon}}\right) \nonumber \\
        &\ll_\epsilon A\sum_{i=1}^N \hspace{-1mm}\sum_{\substack{d<M\\ P^+(d)<i}}  \hspace{-2.5mm}\left(
        \frac{n}{di^2}\exp\left((-1{+}o(1))\frac{\log\frac{n}{i^2d^2}}{\log i} \log\frac{\log\frac{n}{i^2d^2}}{\log i}\right) + \frac{dn^\epsilon}{(id)^{2\epsilon}}\right). \label{eq:e2}
        \end{align}
        We now extend both of the initial two sums in the main term of \eqref{eqn:firstapprox} to infinity, so that we can replace it with $nC_f$, defined in \eqref{eq:cf}. This introduces another error term of the form $O\left(\frac{An}{N}\right)$, to account for all of the new terms $i>N$ in the now infinite sum, and also an error term $E_3$ to account for those additional terms where $i\leq N$ but $d\geq M$.  This gives \[\sum_{i=1}^{N }
        \sum_{\substack{d<M\\ P^+(d)<i}} \sum_{t \in [id,(i+1)d)}
        {\frac{f(d,t)n}{t(t+1)}\prod_{p<i} \frac{p-1}{p}} \ = \ nC_f \ - \ E_3+O\left(\frac{An}{N}\right)\]
        where 
         \begin{align*}
            E_3 = \sum_{i=1}^N  \sum_{\substack{d\geq M\\ P^+(d) \leq i}} \sum_{t \in [id,(i+1)d)} \frac{An}{t(t+1)}\prod_{p\leq i}\frac{p-1}{p} & \leq \sum_{i=1}^N\frac{An}{i^2 } \prod_{p\leq i}\frac{p-1}{p} \sum_{\substack{d\geq M\\ P^+(d)\leq i}} \frac{1}{d} \\
            &= \sum_{i=1}^N\frac{An}{i^2} \prod_{p\leq i}\frac{p-1}{p}S(i,M).
        \end{align*}   
        
        Thus we can now write 
        \begin{equation} \sum_{k=1}^n f(k,n) = nC_f \  + E_1 \ + \ E_2 \ - \ E_3 + O\left(\frac{An}{N}\right) \label{eq:allerrorterms} \end{equation}
        and it remains only to optimize the size of these error terms.

        We start with $E_1$.  Set $B:=\left\lceil\exp((\log \log M)^2)\right\rceil$. We take the sum and remove the initial terms $i<B$ from the sum, which we bound trivially using Lemma \ref{lem:recipsmooth}. (Note that since $B\geq \exp((\log \log M)^2)$, $M\leq \exp \exp \sqrt{\log B}$, satisfying the conditions of that lemma.)
        \begin{align*}
            A\sum_{i=1}^{B-1}
            \Big(\left(\Theta\left(\tfrac{n}{i},i,M\right)-\Theta\left(\tfrac{n}{i},i-1,M\right)\right)\Big) &\ll A\Theta\left(n,B,M\right)   \\
            &= An\exp\left((-1+o(1))\frac{\log M}{\log \log M}\right).
        \end{align*}
        Now, we regroup the terms in the rest of the sum in order to write  
        \begin{align*}
            A&\sum_{i=B}^N \Big(\left(\Theta\left(\tfrac{n}{i},i,M\right)-\Theta\left(\tfrac{n}{i+1},i,M\right)\right)\Big) \\
            &= A\Theta\left(\tfrac{n}{B},B,M\right)+A\hspace{-2mm}\sum_{i=B+1}^N \Big(\Theta\left(\tfrac{n}{i},i,M\right)-\Theta\left(\tfrac{n}{i},i{-}1,M\right)\Big) -A\Theta\left(\tfrac{n}{N+1},N,M\right)\\
            &\leq A\Theta\left(\tfrac{n}{B},B,M\right) +A\hspace{-2mm}\sum_{i=B+1}^N \Big(\Theta\left(\tfrac{n}{i},i,M\right)-\Theta\left(\tfrac{n}{i},i-1,M\right)\Big). 
        \end{align*}
        The first term above is $\ll An\exp\left((-1+o(1))\frac{\log M}{\log \log M}\right)$ for the same argument as before.   %$$\Theta\left(\tfrac{n}{i},i,M\right)-\Theta\left(\tfrac{n}{i},i-1,M\right)$$ 
        The term inside the sum above counts integers having an $i$ smooth divisor greater than $M$, but not an $(i-1)$-smooth divisor, meaning this will be nonzero only if $i$ is prime. When $i=p$ is prime, this counts integers up to $\frac{n}{p}$ divisible by $p$, having a $p$-smooth divisor greater than $M$.  Dividing out the factor of $p$ we find that the contribution from this term will be $\Theta\left(\tfrac{n}{p^2},p,\frac{M}{p}\right)$.  Thus we have 
        \begin{align*}
            A\sum_{i=B+1}^N \Big(\Theta\left(\tfrac{n}{i},i,M\right)&-\Theta\left(\tfrac{n}{i},i-1,M\right)\Big) = A\sum_{B+1\leq p \leq N} \Theta\left(\tfrac{n}{p^2},p,\frac{M}{p}\right) \\
            & \leq A\sum_{i=B+1}^N \Theta\left(\tfrac{n}{i^2},i,\tfrac{M}{i}\right)\\
            &\ll A\sum_{i=B}^N \frac{n}{i^2}\exp\left((-1+o(1))\tfrac{\log M}{\log i}\log\left(\tfrac{\log M}{\log i}\right)\right)
        \end{align*}
        using Lemma \ref{lem:recipsmooth} again.
        
        We treat $E_3$ similarly.  Taking the terms in that sum with $i<B$ we have by Lemma \ref{lem:recipsmooth}
        \begin{align*}
            \sum_{i=1}^{B-1}\frac{An}{i^2} \prod_{p\leq i}\frac{p-1}{p}S(i,M) &\leq An S(B,M)\sum_{i=1}^{B-1}\frac{1}{i^2} \\ &\ll An \log B \exp\left(\frac{(-1+o(1))\log M}{(\log \log M)^2}\log \log M\right)\\
            &= An \exp\left(\frac{(-1+o(1))\log M}{\log \log M}\right).
        \end{align*}
        Also, by Lemma \ref{lem:recipsmooth}, we bound         \begin{align*}
            \sum_{i=B}^{N}\frac{An}{i^2} \prod_{p\leq i}\frac{p-1}{p}S(i,M)  &\ll \sum_{i=B+1}^N \frac{n}{i^2}\exp\left((-1+o(1))\tfrac{\log M}{\log i}\log\left(\tfrac{\log M}{\log i}\right)\right)
        \end{align*}
        Note that both of these bounds are the same ones obtained for $E_1$, so \begin{equation}
            E_1+E_3 \ll \frac{An}{\exp\left(\tfrac{(1{+}o(1))\log M}{\log \log M}\right)} +
         \sum_{i=B+1}^N \tfrac{An}{i^2}\exp\left((-1{+}o(1))\tfrac{\log M}{\log i}\log\tfrac{\log M}{\log i}\right). \label{eq:e1e2bound}
        \end{equation} 
        
        We bound  the sum over $i$ above as 
        \begin{align}
            An&\sum_{i=B+1}^N \frac{1}{i^2}\exp\left((-1{+}o(1))\left( \tfrac{\log M}{\log i}\log\left(\tfrac{\log M}{\log i}\right)\right)\right) \label{eq:inite1e2sum}\\
            &\leq An\left(\sum_{i=B+1}^N \frac{1}{i}\right)\max_{B<i\leq N}\left\{\frac{1}{i}\exp\left((-1{+}o(1))\left(\tfrac{\log M}{\log i}\log\left(\tfrac{\log M}{\log i}\right)\right)\right)\right\}\nonumber \\
            &\leq An\log N\max_{B<i\leq N}\left\{\exp\left((1{+}o(1))\left(-\log i -\tfrac{\log M}{\log i}\log\left(\tfrac{\log M}{\log i}\right)\right)\right)\right\}. \label{eq:maxe1e2}
        \end{align}
        
        Taking the derivative of the quantity in the exponent with respect to $i$, 
        
        $$\frac{\log M - \log^2 i + \log M \log\left(\frac{\log M}{\log i}\right)}{i \log^2 i}$$ and setting the numerator equal to zero, we find that this sum is maximized when 
        $$\log i = \sqrt{\left(\tfrac{1}{2}+o(1)\right)\log M\log \log M}$$
        under the assumption that \begin{equation}\log \log i \leq \log \log N \ll \log \log M. \label{eq:NMassume} \end{equation}
        Using this in \eqref{eq:maxe1e2}, we find this sum is at most
        \begin{align*}
            An\log N \exp&\left((-1{+}o(1))\left(\sqrt{\tfrac{1}{2}\log M\log \log M} +\sqrt{\tfrac{2\log M}{\log \log M}}\log\left(\sqrt{\tfrac{2\log M}{\log \log M}}\right)\right)\right)\\
            &= An\log N \exp\left((-1{+}o(1))\sqrt{2\log M\log \log M} \right)\\
            &= An \exp\left(-\sqrt{(2{+}o(1))\log M\log \log M} \right).
        \end{align*}
        Here the $\log N$ term was absorbed into the $o(1)$ under the assumption of \eqref{eq:NMassume}.  As this is greater than the first term in \eqref{eq:e1e2bound}, we have 
        \begin{equation} E_1+E_3 \ll An \exp\left(-\sqrt{(2{+}o(1))\log M\log \log M} \right). \label{eq:e1e3bd} \end{equation}
        
        For $E_2$ we bound the two components of the sum in \eqref{eq:e2} separately.  Taking first the second term, and assuming $\epsilon\leq \frac{1}{2}$, we first bound

        \begin{align}
            A\sum_{i=1}^N \sum_{\substack{d<M\\ P^+(d)<i}}  \frac{dn^\epsilon}{(id)^{2\epsilon}}&\leq A\sqrt{n}\sum_{i=1}^N \frac{\Psi(M,i)}{i}\nonumber \ll A\sqrt{n} N \max_{i\leq N}\left\{ \frac{\Psi(M,i)}{i} \right\} \nonumber \\ 
            &\ll A\sqrt{n}NM\exp\left(-\sqrt{(2+o(1))\log M \log \log M}\right) \label{eq:e2epsilonbd}
        \end{align}
        using that $\max_{i}\left\{ \frac{\Psi(M,i)}{i}\right\} =\exp\left(-\sqrt{(2+o(1))\log M \log \log M}\right)$ (see \cite{makingsquares}).
        Now we consider the first sum appearing in \eqref{eq:e2}.
        \begin{align}
            &A\sum_{i=1}^N \sum_{\substack{d<M\\ P^+(d)<i}} 
        \frac{n}{di^2}\exp\left((-1{+}o(1))\frac{\log\left(\frac{n}{i^2d^2}\right)}{\log i} \log\left(\frac{\log\left(\frac{n}{i^2d^2}\right)}{\log i}\right) \right) \nonumber \\
        &\ll An\sum_{i=1}^N\left(  \sum_{\substack{d<M\\ P^+(d)<i}}\tfrac{1}{\Psi(d,i)}  \tfrac{\Psi(d,i)}{di^2}\exp\left((-1{+}o(1))\tfrac{\log\left(\frac{n}{d^2}\right)}{\log i} \log\left(\tfrac{\log\left(\frac{n}{d^2}\right)}{\log i}\right) \right) \right)  \nonumber \\
        &\ll An\sum_{i=1}^N\left(  \sum_{\substack{d<M\\ P^+(d)<i}}\hspace{-3mm}\tfrac{1}{\Psi(d,i)}\right)\max_{d<M}\left\{  \tfrac{\Psi(d,i)}{di^2}\exp\left(\tfrac{(-1{+}o(1))\log\left(\frac{n}{d^2}\right)}{\log i} \log\left(\tfrac{\log\left(\frac{n}{d^2}\right)}{\log i}\right) \right)\right\}  \nonumber \\
        &\leq An\log M\sum_{i=1}^N\tfrac{1}{i^2}\max_{d<M}\left\{  \tfrac{\Psi(d,i)}{d}\exp\left((-1{+}o(1))\tfrac{\log\left(\tfrac{n}{d^2}\right)}{\log i} \log\left(\tfrac{\log\left(\tfrac{n}{d^2}\right)}{\log i}\right) \right)\right\}. \label{eq:e2bound}
        \end{align}       
For $i\geq(\log n)^2$ we also have $i \geq (\log d)^2$ and so we can use the upper bound $\Psi(d,i) \ll d\exp\left((-1+o(1))\frac{\log d}{\log i}\log \left(\frac{\log d}{\log i}\right)\right)$, so that

        \begin{align*}
   \frac{\Psi(d,i)}{d}&\exp\left((-1{+}o(1))\frac{\log\left(\frac{n}{d^2}\right)}{\log i} \log\left(\frac{\log\left(\frac{n}{d^2}\right)}{\log i}\right) \right) \\ &\ll \exp\left((-1{+}o(1))\left(\frac{\log d}{\log i}\log \left(\frac{\log d}{\log i}\right)+\frac{\log\left(\frac{n}{d^2}\right)}{\log i} \log\left(\frac{\log\left(\frac{n}{d^2}\right)}{\log i}\right)\right) \right).  \nonumber 
        \end{align*}

    The derivative with respect to $d$ of the main term from the exponent above is $$\frac{2\log\left(\frac{\log \frac{n}{d^2}}{\log i}\right)-\log\left(\frac{\log d}{\log i}\right)+1}{d\log i}.$$
        Setting the numerator equal to 0 and solving for $d$, we find that this expression is maximized when
        \begin{equation} \log d = \left(\frac12+o(1)\right)\log n. \label{eq:logdmax} \end{equation}
Thus we bound
        \begin{align}
            \frac{\Psi(d,i)}{d}\exp\left((-1{+}o(1))\frac{\log\frac{n}{d^2}}{\log i} \log\frac{\log\frac{n}{d^2}}{\log i} \right)& \nonumber \\
            \ll \exp&\left((-1+o(1))\frac{\log n}{2\log i}\log\frac{\log n}{2\log i}\right). \label{eq:largeibd}
        \end{align}
    Inserting this bound into \eqref{eq:e2bound} and summing over $\log^2 n<i<N$, we find that the evaluation of this sum is identical to that of \eqref{eq:inite1e2sum}, and obtain
    \begin{align}
        An\log M&\sum_{\log^2n\leq i\leq N}\frac{1}{i^2}\exp\left(-(1+o(1))\frac{\log n}{2\log i}\log\left(\frac{\log n}{2\log i}\right)\right) \nonumber \\
        &\ll An\log M \exp\left((-1+o(1))\sqrt{\log n \log \log n}\right)\nonumber \\
        &=An\exp\left((-1+o(1))\sqrt{\log n \log \log n}\right). \label{eq:largeibound}
    \end{align}
    On the other hand, if $i<(\log n)^2$, then, since $\log^2 n > \log^2 d$ we bound $$\Psi(d,i)\leq \Psi\left(d,(\log n)^2\right) \ll d \exp\left(-(1+o(1))\frac{\log d}{2\log \log n}\log\left( \frac{\log d}{\log \log n}\right)\right).$$
    
    Using this in \eqref{eq:e2bound}, we find that the exponent is again maximized for $\log d$ satisfying \eqref{eq:logdmax}, and we get
    
    \begin{align}
   &\frac{\Psi(d,i)}{d}\exp\left((-1{+}o(1))\frac{\log\left(\frac{n}{d^2}\right)}{\log i} \log\left(\frac{\log\left(\frac{n}{d^2}\right)}{\log i}\right) \right) \nonumber \\ 
   &\ll \exp\left((-1{+}o(1))\left(\frac{\log d}{2\log \log n}\log\left( \frac{\log d}{\log \log n}\right)+\frac{\log\left(\frac{n}{d^2}\right)}{\log i} \log\left(\frac{\log\left(\frac{n}{d^2}\right)}{\log i}\right)\right) \right)  \nonumber \\
    &\ll \exp\left((-1{+}o(1))\left(\frac{\log d}{2\log \log n}\log\left( \frac{\log d}{\log \log n}\right)+\frac{\log\left(\frac{n}{d^2}\right)}{2\log \log n} \log\left(\frac{\log\left(\frac{n}{d^2}\right)}{2\log \log n}\right)\right) \right)  \nonumber \\
    &\ll \exp\left((-\tfrac{1}{4}+o(1))\log n\right).  \nonumber 
    \end{align}
    Inserting this in \eqref{eq:e2bound} and summing over $i<\log^2 n$, we find this contributes at most
    \begin{align*}
        An\log M\sum_{1\leq i < \log^2n} \frac{1}{i^2}\exp\left((-\tfrac{1}{4}+o(1))\log n\right) \ll An^{3/4 +o(1)}.
    \end{align*}
    Which is dominated by \eqref{eq:largeibound}.  Thus, combining \eqref{eq:e2epsilonbd} and \eqref{eq:largeibound} we have 
    \[E_2 \ll \frac{A\sqrt{n}NM}{\exp\left(\sqrt{(2{+}o(1))\log M \log \log M}\right)}+ \frac{An}{\exp\left((1{+}o(1))\sqrt{\log n \log \log n}\right)}.\]
    Combined with the upper bound \eqref{eq:e1e3bd} for $E_1+E_3$ and inserting into \eqref{eq:allerrorterms}, we get
    
    \begin{align*}
        \sum_{k=1}^n f(k,n) - nC_f \ &= \ E_1 \ + \ E_2 \ + \ E_3 + O\left(\frac{An}{N}\right) \\
        \ll An \exp&\left(-\sqrt{(2{+}o(1))\log M\log \log M} \right) \\ &+ A\sqrt{n}NM\exp\left(-\sqrt{(2{+}o(1))\log M \log \log M}\right)\\
        &+ An\exp\left((-1{+}o(1))\sqrt{\log n \log \log n}\right) + \frac{An}{N}.
    \end{align*}
    Optimizing this, subject to the conditions of \eqref{ineq:utinfty} and \eqref{eq:NMassume}, we take $\log N=\sqrt{\log n \log \log n}$ and $\log M=\tfrac{1}{2}\log n -\sqrt{\log n} \log \log n$.  This gives us that the expression above is 
    $$\ll An \exp\left((-1+o(1))\sqrt{\log n \log \log n}\right)$$ as desired.
\end{proof}
\section*{Acknowledgements}
The author is grateful to the anonymous referees whose comments helped improve constant appearing in the exponent of the main theorem, as well as many other helpful changes.  

\bibliographystyle{amsplain}
\bibliography{bibliography}

\appendix 
\section{Ranges of values used in computations}

\begin{table}[h!]
    \centering
    \begin{tabular}{|c|c|}
    \hline
        $i$&$d$ up to \\
        \hline
        12 & 6144 \\
        $13\leq i \leq15$ & 1536 \\
        $16\leq i \leq20$ & 1152 \\
        $21\leq i \leq24$ & 256\\
        25 &150\\
        $26\leq i \leq$30 & 16\\
        $31\leq i \leq$40 & 12\\
        $41\leq i \leq$55 &8\\
        $56\leq i \leq$60 &7\\
        $61\leq i \leq$75 &4\\
        $76\leq i \leq$100 &3\\
        $101\leq i \leq$250 &1\\
        \hline
    \end{tabular}
    \caption{Values of $g(d,t)$ were computed for all $d$ up to and including the value in the table for each value of $i$ in order to estimate the constant $\theta$ in the count of geometric progression free sets.}
    \label{tab:gpf}
\end{table}
\end{document}